\documentclass{article}
\usepackage[T1]{fontenc}
\usepackage{tgtermes}
\usepackage{mathptmx}  
\usepackage[all]{xy}
\usepackage[scaled=.92]{helvet}
\usepackage{amsthm,amsmath,amssymb}
\usepackage{mathrsfs}
\usepackage[numbers]{natbib}  
\usepackage[colorlinks,citecolor=blue,urlcolor=blue]{hyperref}  

\usepackage{enumerate}

\newtheorem{theorem}{Theorem}[section]
\newtheorem{lemma}[theorem]{Lemma}

\newtheorem{cor}[theorem]{Corollary} 
\newtheorem{proposition}[theorem]{Proposition} 
\theoremstyle{definition}
\newtheorem{definition}[theorem]{Definition}
\newtheorem{example}[theorem]{Example}

\theoremstyle{remark}
\newtheorem{remark}[theorem]{Remark}

\date{}

  \title{Remarks on abstract structures \\of propositions and realizers}

  \author{Samuele Maschio
  }
  
\begin{document}
\maketitle
\begin{abstract}
We present here an abstract notion of structure consisting of proposi-
tions and realizers (which we call PR-structures) giving rise to set based
contravariant functors taking values in the category of sets endowed with
binary relations. We will characterize those PR-structures giving rise to
preorderal and posetal doctrines and we will study in particular the case
of a PR-structure induced by a partial applicative structure.
\end{abstract}

\section{Introduction}
Every topos of Boolean or Heyting valued sets and essentially every topos which is known under the name of \emph{``whatever realizability'' topos} can be obtained as the result of a tripos-to-topos construction (see e.g.\ \cite{tripos}) based on the category of sets. However, these triposes have a very specific aspect: they all involve some abstract notion of proposition and/or of realizer. E.g.\ in the case of the tripos of Heyting-valued sets \cite{higgs1973category}, the role of propositions is taken by a complete Heyting algebra; in the case of the tripos giving rise to the effective topos (see \cite{eff}), the realizers are natural numbers while propositions are identified with subsets of $\mathbb{N}$. A very general construction combining together realizers and propositions was recently proposed by A.Miquel in \cite{Miquel2018ImplicativeAA}.

Here we follow a different direction. We will study very general notions of realizer and of proposition combined together to give rise to a contravariant functors from $\mathsf{Set}$ to the category of sets endowed with a binary relation. For this aim, we will introduce the notion of PR-structure and we will find characterizations for those PR-structures giving rise to preoderal and posetal doctrines. In particular, in the last part of the paper, we will focus on doctrines coming from PR-structures defined using a partial applicative structure. 

We will take Zermelo-Fraenkel set theory with axiom of choice $\mathsf{ZFC}$ to be our metatheory, however we will use the axiom of choice only in the proof of theorem \ref{can}.

%
%
%
%
%
%
%
%
%

\section{PR-structures}
We give here the definition of PR-structure and then we will provide some examples.
\begin{definition} A \emph{PR-structure} is a triple $\Sigma=(\mathsf{P}_\Sigma,\mathsf{R}_\Sigma,\rho^\Sigma)$ where $\mathsf{P}_\Sigma$ and $\mathsf{R}_\Sigma$ are non-empty sets and $\rho^\Sigma$ is a function from $\mathsf{P}_\Sigma\times \mathsf{P}_\Sigma$ to $\mathcal{P}(\mathsf{R}_\Sigma)$.
\end{definition}
\begin{definition} Let $\Sigma$ be a PR-structure. 
 If $I$ is a set and $\varphi,\psi:I\rightarrow \mathsf{P}_\Sigma$ are functions, then the relation $\varphi\vdash^{\Sigma}_I \psi$ holds if and only if $\bigcap_{i\in I}\rho^\Sigma(\varphi(i),\psi(i))\neq \emptyset$. We also define a binary relation $\vdash^{\Sigma}$ on $\mathsf{P}_\Sigma$ as follows: for $a,b\in \mathsf{P}_\Sigma$, $a\vdash^\Sigma b$ if and only if $\rho^\Sigma(a,b)\neq \emptyset$. 
\end{definition}
\begin{definition}
We denote with $\mathsf{Bin}$ the category whose objects are pairs $(A,R)$ consisting of a set $A$ and a binary relation $R$ on it. 
An arrow in $\mathsf{Bin}$ from $(A,R)$ to $(B,S)$ is a function $f:A\rightarrow B$ such that for every $a,a'\in A$, if $R(a,a')$ then $S(f(a),f(a'))$. Compositions and identities are the set-theoretical ones.
\end{definition}

\begin{proposition}The assignments $I\mapsto (\mathsf{P}_\Sigma^{I},\vdash^{\Sigma}_I)$ and $f\mapsto (-)\circ f$ 
define a contravariant functor $\mathbf{p}_\Sigma$ from $\mathsf{Set}$ to $\mathsf{Bin}$.
\end{proposition}
\begin{proof}
Suppose that $f:I\rightarrow J$ is a function and that $\varphi,\psi\in \mathsf{P}_\Sigma^J$ satisfy $\varphi\vdash^{\Sigma}_{J}\psi$. Then,
$$\bigcap_{i\in I}\rho^\Sigma ((\varphi\circ f)(i),(\psi\circ f)(i))=\bigcap_{i\in I}\rho^\Sigma (\varphi(f(i)),\psi (f(i)))\supseteq \bigcap_{j\in J}\rho^\Sigma (\varphi(j),\psi (j))\neq \emptyset$$
Thus, $\varphi\circ f\vdash^{\Sigma}_{I}\psi\circ f$.
\end{proof}
\begin{definition} Two PR-structures $\Sigma$ and $\Sigma'$ are said to be \emph{equivalent}, and in this case we write $\Sigma\approx\Sigma'$, if $\mathsf{P}_\Sigma=\mathsf{P}'_\Sigma$ and $\mathbf{p}_\Sigma =\mathbf{p}_{\Sigma'} $.

We define the degree $\delta(\Sigma)$ of a PR-structure $\Sigma$ as $\min\{|\mathsf{R}_{\Sigma'}|\,|\,\Sigma'\approx\Sigma \}$. We say that $\Sigma$ is a \emph{P-structure} if $\delta(\Sigma)=1$ and we say that $\Sigma$ is \emph{finite} if $\delta(\Sigma)$ is finite.
\end{definition}

\begin{definition}
A PR-structure $\Sigma$ is called \emph{partitioned} if $|\rho^\Sigma(a,b)|\leq 1$ for all $a,b\in \mathsf{P}_\Sigma$.
\end{definition}

\begin{definition}
Let $\Sigma$ be a PR-structure and let $r\in \mathsf{R}_\Sigma$. Then we define 
$$\rho^{\Sigma-}(r):=\{(a,b)\in \mathsf{P}\times \mathsf{P}|\,r\in \rho^\Sigma(a,b)\}$$
\end{definition}
The following results immediately follow from the definitions.
\begin{proposition}
A PR-structure $\Sigma$ is partitioned if and only if $\rho^{\Sigma -}(r)\cap \rho^{\Sigma -}(s)=\emptyset$ for every $r,s\in \mathsf{R}_\Sigma$ with $r\neq s$. 
\end{proposition}

\begin{proposition}\label{pointwise} If  $\Sigma$ is a PR-structure and $I$ is a set, then $\varphi\vdash^{\Sigma}_{I}\psi$ implies $\varphi(i) \vdash^{\Sigma} \psi(i)$ for every $i\in I$.\end{proposition}
\begin{proposition}\label{singleton} If $\Sigma$ is a PR-structure and $|I|=1$, then $(\mathsf{P}_\Sigma,\vdash^{\Sigma})$ is isomorphic to $(\mathsf{P}_\Sigma^I,\vdash_{I}^{\Sigma})$ in $\mathsf{Bin}$.
\end{proposition}

\begin{example}Here are some examples of PR-structure.
\begin{enumerate}
\item If $\mathcal{P}=(|\mathcal{P}|,\leq)$ is an object of the category $\mathsf{Bin}$, we define the PR-structure $\Sigma[\mathcal{P}]$ as $(|\mathcal{P}|,\{*\},\chi_\leq)$ where
$$
\begin{cases}
\chi_\leq(a,b):=\{*\}\,\textrm{ if }a\leq b\\
\chi_\leq (a,b):=\emptyset \textrm{ otherwise}\\
\end{cases}
$$
If $\mathcal{P}$ is a complete Heyting algebra, $\mathbf{p}_{\Sigma[\mathcal{P}]}$ is the tripos giving rise to a topos of Heyting-valued sets (see \cite{higgs1973category}).

\item If $\mathcal{R}=(|\mathcal{R}|,\cdot_\mathcal{R})$ is a \emph{partial applicative structure}, that is a non-empty set $|\mathcal{R}|$ together with a partial binary function $\cdot_\mathcal{R}:|\mathcal{R}|\times |\mathcal{R}|\rightharpoonup |\mathcal{R}|$, we can consider the PR-structure $\Sigma[\mathcal{R}]:=(\mathcal{P}(|\mathcal{R}|),|\mathcal{R}|,\Rightarrow^\mathcal{R})$ where for $A,B\subseteq |\mathcal{R}|$
$$A\Rightarrow^\mathcal{R} B:=\{r\in |\mathcal{R}||\,\textrm{ for every }a\in A(r\cdot_\mathcal{R} a\downarrow\textrm{ and }r\cdot_\mathcal{R} a\in B)\}.$$
If $\mathcal{R}$ is a partial combinatory algebra (for short pca, see e.g.\ \cite{VOO08}), then $\Sigma[\mathcal{R}]$ is the tripos giving rise to the realizability topos $\mathsf{RT}[\mathcal{R}]$. In particular, if $\mathcal{R}$ is the first Kleene algebra $\mathcal{K}_1$, then $\Sigma[\mathcal{K}_1] $ is the tripos giving rise to the effective topos $\mathcal{E}ff$ (see \cite{eff}).
\item If $\mathcal{A}_\#$ is a subpca of a pca $\mathcal{A}$, then we can define a PR-structure $\Sigma[\mathcal{A}_\#,\mathcal{A}]$ as follows:
\begin{enumerate}
\item $\mathsf{P}_{\Sigma[\mathcal{A}_\#,\mathcal{A}]}:=\{(I,J)\in \mathcal{P}(|\mathcal{A}_\#|)\times  \mathcal{P}(|\mathcal{A}|)|\,I\subseteq J\}$
\item $\mathsf{R}_{\Sigma[\mathcal{A}_\#,\mathcal{A}]}:=|\mathcal{A}_\#|$
\item $\rho^{\Sigma[\mathcal{A}_\#,\mathcal{A}]}((I,J),(I',J'))=(I\rightarrow^\mathcal{A} I')\cap (I\rightarrow^\mathcal{A_\#} I')\cap |\mathcal{A}_\#|$
\end{enumerate}
The doctrines $\mathbf{p}_{\Sigma[\mathcal{A}_\#,\mathcal{A}]}$ are exactly the triposes giving rise to nested realizability toposes (see \cite{MR3325573} and \cite{MR1933398}).
\item If $\mathcal{A}_\#$ is a subpca of a pca $\mathcal{A}$, then we can define a PR-structure $\Sigma_{rel}[\mathcal{A}_\#,\mathcal{A}]$ as follows:
\begin{enumerate}
\item $\mathsf{P}_{\Sigma_{rel}[\mathcal{A}_\#,\mathcal{A}]}:= \mathcal{P}(|\mathcal{A}|)$
\item $\mathsf{R}_{\Sigma_{rel}[\mathcal{A}_\#,\mathcal{A}]}:=|\mathcal{A}_\#|$
\item $\rho^{\Sigma_{rel}[\mathcal{A}_\#,\mathcal{A}]}(I,I')=(I\rightarrow^\mathcal{A} I')\cap |\mathcal{A}_\#|$
\end{enumerate}
The doctrine $\mathbf{p}_{\Sigma_{rel}[\mathcal{A}_\#,\mathcal{A}]}$ are exactly the triposes giving rise to relative realizability toposes (see  \cite{MR1933398}).
\item Similarly one can also produce PR-structures for which the relative doctrines are the triposes giving rise to modified relative realizability toposes (see \cite{MR1933398}) and to classical realizability toposes (see \cite{MR3150099} and \cite{MR2675258}).
\end{enumerate}
\end{example}

\subsection{A canonical representation for PR-structures}

The following lemma is an immediate consequence of the definition of the notion of equivalence between PR-structures.

\begin{lemma}\label{lemma1} Let $\Sigma$ be a PR-structure and suppose $\overline{r}\in \mathsf{R}_\Sigma$ and $S\subseteq \mathsf{R}_\Sigma\setminus \{\overline{r}\}$ satisfy $\rho^{\Sigma-} (s)\subseteq \rho^{\Sigma-} (\overline{r})$ for every $s\in S$.
Then 
$$\Sigma\approx (\mathsf{P}_\Sigma,\mathsf{R}_\Sigma \setminus S,(a,b)\mapsto \rho(a,b)\setminus S).$$

\end{lemma}
%
%

We can now prove that every PR-structure has a canonical representation.

\begin{theorem}	\label{can} Every PR-structure $\Sigma$ is equivalent to a PR-structure of the form $(\mathsf{P}_\Sigma,\mathcal{I},\epsilon)$ in which 
\begin{enumerate}
\item $\mathcal{I}$ is an antichain in $(\mathcal{P}(\mathsf{P}_\Sigma\times \mathsf{P}_\Sigma),\subseteq)$, that is for every $I,J\in \mathcal{I}$, if $I\subseteq J$, then $I=J$.

\item $\epsilon(x,y)=\{I\in \mathcal{I}|\,(x,y)\in I\}$ for every $x,y\in \mathsf{P}_\Sigma$.
\end{enumerate}
\end{theorem}
\begin{proof} Let $\Sigma$ be a $PR$-structure. Using the well-ordering theorem we can enumerate the elements of $\mathsf{R}_\Sigma$ using an ordinal $\eta$, obtaining $\mathsf{R}_\Sigma=\{r_\xi|\,\xi<\eta\}$. We then define a transfinite sequence of subsets of $\mathsf{R}_\Sigma$ as follows:
\begin{enumerate}
\item $\mathsf{R}_\Sigma ^0:=\mathsf{R}\setminus \{x\in \mathsf{R}|\ x\neq r_0\wedge \rho^{\Sigma-} (x)\subseteq \rho^{\Sigma-} (r_0)\}$
\item $\mathsf{R}_\Sigma ^\xi:=\left(\bigcap_{\xi'<\xi}\mathsf{R}_\Sigma^{\xi'}\right)\setminus \{x\in \mathsf{R}|\ x\neq r_\xi\wedge \rho^{\Sigma-} (x)\subseteq \rho^{\Sigma-} (r_\xi)\}$ for ordinals $0<\xi<\eta$
\end{enumerate}
and we define $\widetilde{\mathsf{R}_\Sigma} $ as $\bigcap_{\xi<\eta}\mathsf{R}_\Sigma^{\xi}$. Using lemma \ref{lemma1} and transfinite induction, one obtains that $(\mathsf{P},\widetilde{\mathsf{R}_{\Sigma}},(a,b)\mapsto \rho(a,b)\cap \widetilde{\mathsf{R}_{\Sigma}})\approx\Sigma $.
But $$(\mathsf{P},\widetilde{\mathsf{R}_{\Sigma}},(a,b)\mapsto \rho(a,b)\cap \widetilde{\mathsf{R}_{\Sigma}})\approx (\mathsf{P},\rho^{\Sigma-} (\widetilde{\mathsf{R}_{\Sigma}}),\epsilon)$$ and $\rho^{\Sigma-} (\widetilde{\mathsf{R}_{\Sigma}})$ is an antichain in $(\mathcal{P}(\mathsf{P}_\Sigma\times \mathsf{P}_\Sigma),\subseteq)$, by definition of $\widetilde{\mathsf{R}_\Sigma}$.
\end{proof}

One can notice that the PR-structure $(\mathsf{P},\rho^{\Sigma-} (\widetilde{\mathsf{R}_{\Sigma}}),\epsilon)$ in the previous lemma does not depend on the particular enumeration of $\mathsf{R}_\Sigma$ and thus it is a canonical representation. From the proof of the theorem above some corollaries follow:

\begin{cor} $\Sigma_1\approx\Sigma_2$ if and only if $\mathsf{P}_{\Sigma_1}=\mathsf{P}_{\Sigma_2}$ and $\rho^{\Sigma_1-}(\widetilde{\mathsf{R}_{\Sigma_1}} )=\rho^{\Sigma_2-}(\widetilde{\mathsf{R}_{\Sigma_2}} )$. 
\end{cor}

\begin{cor} If  $\Sigma$ is a PR-structure such that $\mathsf{P}_\Sigma$ is finite, then $\Sigma$ is finite.
\end{cor}
\begin{proof} As proved in theorem \ref{can}, $\Sigma \approx (\mathsf{P},\rho^{\Sigma-} (\widetilde{\mathsf{R}_{\Sigma}}),\epsilon)$, but $\rho^{\Sigma-} (\widetilde{\mathsf{R}_{\Sigma}})\subseteq \mathcal{P}(\mathsf{P}_\Sigma)$ and $\mathsf{P}_\Sigma$ is finite; thus $\rho^{\Sigma-} (\widetilde{\mathsf{R}_{\Sigma}})$ is finite.
\end{proof}

%

In the next proposition we characterize P-structures.

\begin{proposition} A PR-structure $\Sigma$ is a P-structure if and only if, for every $I$ and for every $\varphi,\psi:I\rightarrow \mathsf{P}_\Sigma$, $\varphi\vdash^{\Sigma}_{I}\psi$ if and only if $\varphi(i)\vdash^{\Sigma}\psi(i)$ for every $i\in I$.
\end{proposition}

\begin{proof} Let us prove the two directions of the equivalence.
\begin{enumerate}
\item[($\Rightarrow$)] Suppose $\Sigma$ is a PR-structure with $\mathsf{R}_\Sigma=\{r\}$ and let $\varphi,\psi:I\rightarrow \mathsf{P}_\Sigma$. Obviously, 
$\varphi\vdash^{\Sigma}_{I}\psi$ if and only if $\bigcap_{i\in I}\rho(\varphi(i),\psi(i))=\{r\}$ if and only if $\rho(\varphi(i),\psi(i))=\{r\}$ for every $i\in I$, that is, $\varphi(i)\vdash^{\Sigma}\psi(i)$ for every $i\in I$.

\item[($\Leftarrow$)] Suppose that $\Sigma$ is a PR-structure such that for every $I$ and for every $\varphi,\psi:I\rightarrow \mathsf{P}_\Sigma$, $\varphi\vdash^{\Sigma}_{I}\psi$ if and only if $\varphi(i)\vdash^{\Sigma}\psi(i)$ for every $i\in I$. Consider $\pi_1,\pi_2:\,\vdash^{\Sigma}\rightarrow \mathsf{P}_\Sigma$. Since for every $(a,b)\in\, \vdash^\Sigma $ we have 
$$\pi_{1}((a,b))=a\vdash^{\Sigma}b=\pi_{2}((a,b)),$$ 
then, by our assumption, $\pi_1\vdash^{\Sigma}_{\vdash^{\Sigma}}\pi_{2}$. Hence, by definition, there exists $\overline{r}\in \mathsf{R}_\Sigma$ such that $\overline{r}\in \rho(a,b)$ for every $a,b\in \mathsf{P}_\Sigma$ such that $a\vdash^{\Sigma}b$. In particular, for every  $s\in \mathsf{R}_\Sigma$, $\rho^{\Sigma-} (s)\subseteq \rho^{\Sigma-} (\overline{r})$.
Thus, as a consequence of lemma \ref{lemma1}, $\Sigma$ is equivalent to 
$$(\mathsf{P}_\Sigma,\{\overline{r}\},(a,b)\mapsto \rho(a,b)\cap \{\overline{r}\})$$
and it is hence a P-structure.
\end{enumerate}
\end{proof}

\begin{remark}One can always produce examples of $PR$-structures which behave like a $P$-structure, but only up to some cardinality. Consider a binary relation $\Psi $ on a set $\mathsf{P}$ and the PR-structures $(\mathsf{P},\{J\subseteq \Psi |\,|J|<n\},\epsilon)$ (for $n<|\Psi|$ a natural number) and $(\mathsf{P},\{\Psi \setminus \{(a,b)\}|\,(a,b)\in\Psi\},\epsilon)$. 
In the first case, for every $I$ with $|I|<n$ and for every $\varphi,\psi:I\rightarrow \mathsf{P}$, $\varphi\vdash^{\Sigma}_{I}\psi$ if and only if $\varphi(i)\vdash^{\Sigma}\psi(i)$ for every $i\in I$. In the second case, for every $I$ with $|I|<|\Psi|$ and for every $\varphi,\psi:I\rightarrow \mathsf{P}$, $\varphi\vdash^{\Sigma}_{I}\psi$ if and only if $\varphi(i)\vdash^{\Sigma}\psi(i)$ for every $i\in I$.

\end{remark}


\subsection{Preorderal and posetal PR-structures}

Here we characterize those PR-structures giving rise to set-indexed preorders and posets.

\begin{definition} A PR-structure $\Sigma$ is \emph{preorderal} if $\mathbf{p}_\Sigma$ factors through the subcategory $\mathbf{J}:\mathsf{PreOrd}\rightarrow\mathsf{Bin}$, that is, if for every set $I$, $\mathbf{p}_{\Sigma}(I)$ is a preordered set.

\end{definition}

\begin{theorem}\label{preorderal} A PR-structure $\Sigma$ is preorderal if and only if the following two conditions are satisfied:
\begin{enumerate}
\item there exists $\mathbf{i}\in \mathsf{R}_\Sigma$ such that $\mathbf{i}\in \rho^\Sigma(a,a)$ for every $a\in \mathsf{P}_\Sigma$
\item for every $r,s\in \mathsf{R}_\Sigma$ there exists (but in general is not unique) $s\Box r\in \mathsf{R}_\Sigma$ such that for every $a,b,c\in \mathsf{P}_\Sigma$, $s\Box r\in \rho^\Sigma(a,c)$ whenever $r\in \rho^\Sigma(a,b)$ and $s\in \rho^\Sigma(b,c)$.
\end{enumerate}
\end{theorem}
\begin{proof}
We prove the two directions of the equivalence.
\begin{enumerate}
\item[($\Rightarrow$)] Suppose $\Sigma$ is preorderal. Then $\vdash^{\Sigma}_{\mathsf{P}_\Sigma}$ is reflexive. This means in particular that $\mathsf{id}_{\mathsf{P}_\Sigma}\vdash_{\mathsf{P}_\Sigma}^{\Sigma}\mathsf{id}_{\mathsf{P}_\Sigma}$, that is, there exists $\mathbf{i}\in \mathsf{R}_\Sigma $ such that 
$$\mathbf{i}\in \bigcap_{a\in \mathsf{P}_\Sigma}\rho^\Sigma(a,a).$$
If $r,s\in \mathsf{R}_\Sigma $, then $\vdash^{\Sigma}_{\mathsf{P}_{r,s}}$ is transitive, where 
$$\mathsf{P}_{r,s}:=\{(a,b,c)\in \mathsf{P}_\Sigma\times \mathsf{P}_\Sigma\times \mathsf{P}_\Sigma|\, r\in \rho^\Sigma(a,b), s\in \rho^\Sigma(b,c)\}.$$
Since $\pi_{1}\vdash^{\Sigma}_{\mathsf{P}_{r,s}}\pi_{2}$ and $\pi_{2}\vdash^{\Sigma}_{\mathsf{P}_{r,s}}\pi_{3}$, then $\pi_{1}\vdash^{\Sigma}_{\mathsf{P}_{r,s}}\pi_{3}$. This means that there exists $s\Box r\in \mathsf{R}_\Sigma $ such that 
$$s\Box r\in \bigcap_{a,b,c\in \mathsf{P}_\Sigma, r\in \rho^\Sigma(a,b), s\in \rho^\Sigma(b,c)}\rho^\Sigma(a,c) $$

\item[($\Leftarrow$)] Conversely, suppose $I$ is a set. Let $\varphi:I\rightarrow \mathsf{P}_\Sigma$. Since $\mathbf{i}\in \bigcap_{i\in I}\rho^\Sigma(\varphi(i),\varphi(i))$, then $\varphi\vdash^{\Sigma}_{I}\varphi$. Let $\varphi,\psi,\eta:I\rightarrow \mathsf{P}_\Sigma$ such that $\varphi\vdash^{\Sigma}_{I}\psi$ and $\psi\vdash^{\Sigma}_{I}\eta$. Then there exist $r,s$ such that $r\in \bigcap_{i\in I}\rho^\Sigma(\varphi(i),\psi(i))$ and $s\in \bigcap_{i\in I}\rho^\Sigma(\psi(i),\eta(i))$. For such $r$ and $s$ there exists $s\Box r\in \mathsf{R}_\Sigma$ such that $s\Box r\in \bigcap_{i\in I}\rho^\Sigma(\varphi(i),\eta(i))$; thus $\varphi\vdash^{\Sigma}_{I}\eta$.
\end{enumerate}
\end{proof}


\begin{definition} A PR-structure $\Sigma$ is \emph{posetal} if $\mathbf{p}_\Sigma$ factors through the subcategory $\mathbf{J}:\mathsf{Pos}\rightarrow\mathsf{Bin}$, that is, if for every set $I$, $\mathbf{p}_{\Sigma}(I)$ is a partially ordered set.
\end{definition}
\begin{theorem} A PR-structure $\Sigma$ is posetal if and only it is preorderal and $\vdash^{\Sigma}$ is antisymmetric.\end{theorem}
\begin{proof}
We prove the two directions of the equivalence.
\begin{enumerate}
\item[($\Rightarrow$)] Since $(\mathsf{P}_\Sigma,\vdash^{\Sigma})$ is isomorphic to $(\mathsf{P}_\Sigma^{\{*\}},\vdash^{\Sigma}_{\{*\}})$ in $\mathsf{Bin}$ (by proposition \ref{singleton}), the consequence is immediate.
\item[($\Leftarrow$)] Suppose $\Sigma$ is preorderal and $\vdash^{\Sigma}$ is antisymmetric. Let $\varphi,\psi:I\rightarrow \mathsf{P}_\Sigma$ such that $\varphi\vdash^{\Sigma}_{I}\psi$ and $\psi\vdash_{I}^{\Sigma}\varphi$.  Then for every $i\in I$, $\varphi(i)\vdash^{\Sigma}\psi(i)$ and $\psi(i)\vdash^{\Sigma}\varphi(i)$. Thus, $\varphi(i)=\psi(i)$ for every $i\in I$, that is, $\varphi=\psi$.
\end{enumerate}
\end{proof}

%
Here follow some corollaries of the previous two theorems.
\begin{cor}
For every positive $n\in \mathbb{N}$, there exists a posetal $PR$-structure $\Sigma$ such that $\delta(\Sigma)=n$.
\end{cor}
\begin{proof}
Let $n\in \mathbb{N}$ and consider the PR-structure $\Sigma_n:=(\{1,...,n\},\{1,...,n\},\rho_{n})$
where $\rho_{n}$ is defined as follows $\rho_n(i,j):=\{x\in \mathbb{N}|\,i=j\leq x\leq n\textrm{ or }i=x<j\}$.

\noindent E.g.\ for $n=3$ we can represent the PR-structure as follows:
$$
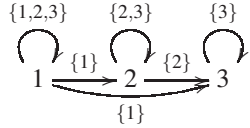
\xymatrix{
1\ar@(ul,ur)^{\{1,2,3\}}\ar[r]^{\{1\}}\ar@/_/[rr]_{\{1\}}	&2\ar@(ul,ur)^{\{2,3\}}\ar[r]^{\{2\}}		&3\ar@(ul,ur)^{\{3\}}\\
}$$
For every $n$, this is a posetal PR-structure. Indeed, one can define $\mathbf{i}$ in theorem \ref{preorderal} as $n$, while for $i,j\in \{1,...,n\}$, we can take $i\Box j$ (as in theorem \ref{preorderal}) to be $\min(i,j)$; finally $\vdash^{\Sigma_n}$ is clearly antisymmetric since it is the usual order of natural numbers on $\{1,...,n\}$. It is clear from the definition that $\delta(\Sigma)=n$. 

\end{proof}

\begin{cor}\label{part}
If $\Sigma$ is a partitioned preorderal PR-structure, then, following notation in proposition \ref{preorderal}, $\mathbf{i}$ is unique and $s\Box r$ is unique for every $r,s\in \mathsf{R}_\Sigma$ and 
 $$\left(\bigcup_{a,b\in\mathsf{P}_\Sigma}\rho(a,b),\Box, \mathbf{i}\right)$$ is a monoid.
\end{cor}

\begin{cor}
Suppose $\mathcal{P}$ is an object of $\mathsf{Bin}$. The PR-structure 
$\Sigma[\mathcal{P}]$ defined in Example 1.1 is preorderal (resp.\ posetal) if and only if $\mathcal{P}$ is a preordered set (resp.\ a partially ordered set).
\end{cor}

\subsection{Bounded posetal PR-structures}
The category $\mathsf{bPos}$ has as objects posets having a minimum and a maximum and as arrows monotone maps preserving minima and maxima. A PR-structure is \emph{bounded-posetal} if $\mathbf{p}_\Sigma$ factors through the subcategory  $\mathbf{J}:\mathsf{bPos}\rightarrow \mathsf{Bin}$, that is, for every set $I$, $\mathbf{p}_{\Sigma}(I)$ is a poset having a minimum and a maximum and for every function $f$, $\mathbf{p}_{\Sigma}(f)$ preserves them.
\begin{theorem}\label{bound}
A PR-structure $\Sigma$ is bounded-posetal if and only if it is posetal and there exist $\bot,\top\in \mathsf{P}_\Sigma$ and $\mathbf{b},\mathbf{t}\in \mathsf{R}_\Sigma$ such that for every $a\in \mathsf{P}_\Sigma$, $\mathbf{b}\in \rho^{\Sigma}(\bot,a)$ and $\mathbf{t}\in \rho^{\Sigma}(a,\top)$.
\end{theorem}
\begin{proof} We prove the two directions of the equivalence.
\begin{enumerate}
\item[$(\Rightarrow)$] Let $\Sigma$ be bounded-posetal and suppose $\bot_{\mathsf{P}_\Sigma}$ and $\top_{\mathsf{P}_\Sigma}$ are the minimum and maximum, respectively, in $(\mathsf{P}_\Sigma^{\mathsf{P}_\Sigma},\vdash^{\Sigma}_{\mathsf{P}})$. Since for every $b\in \mathsf{P}_\Sigma$ we have $\bot_{\mathsf{P}_\Sigma}\vdash^{\Sigma}_{\mathsf{P}_\Sigma}k_b$, where $k_b$ is the constant function with value $b$, for every $a,b\in \mathsf{P}_\Sigma$, $\bot_{\mathsf{P}_\Sigma}(a)\vdash^{\Sigma}b$. Thus, for every $a\in \mathsf{P}_\Sigma$, $\bot_{\mathsf{P}_\Sigma}(a)$ must be equal to the unique minimum $\bot $ of $(\mathsf{P}_\Sigma,\vdash^\Sigma )$. Moreover, since $\bot_{\mathsf{P}_\Sigma}\vdash^{\Sigma}_{\mathsf{P}_\Sigma}\mathsf{id}_{\mathsf{P}_\Sigma}$, there exists $\mathbf{b}\in \mathsf{P}_\Sigma$ such that $\mathbf{b}\in \rho^{\Sigma}(\bot,a)$ for every $a\in \mathsf{P}_\Sigma$. An analogous proof works for maxima.
\item[$(\Leftarrow)$] Suppose that there exist $\bot,\top\in \mathsf{P}_\Sigma$ and $\mathbf{b},\mathbf{t}\in \mathsf{R}_\Sigma$ such that for every $a\in \mathsf{P}_\Sigma$, $\mathbf{b}\in \rho^{\Sigma}(\bot,a)$ and $\mathbf{t}\in \rho^{\Sigma}(a,\top)$ and let $I$ be a set. Consider the constant functions $\bot_I$ and $\top_I$ from $I$ to $\mathsf{P}_\Sigma$ defined as $\bot$ and $\top$ on any entry, respectively. For every $\varphi:I\rightarrow \mathsf{P}_\Sigma$ we clearly have that $\mathbf{b}\in \bigcap_{i\in I}\rho^{\Sigma}(\bot_I(i),\varphi(i))$ and $\mathbf{t}\in \bigcap_{i\in I}\rho^{\Sigma}(\varphi(i),\top_I(i))$; thus $\bot_I\vdash^\Sigma _I\varphi\vdash^\Sigma _I\top_I$. So $\bot_I$ is a minimum and $\top_I$ is a maximum in $(\mathsf{P}_\Sigma^I,\vdash^\Sigma _I)$. Since constant functions are preserved by precomposition, minima and maxima are so. Thus $\Sigma$ is bounded-posetal.
\end{enumerate}
\end{proof}

Here we present two sufficient conditions for concluding that a bounded-posetal PR-structure is a P-structure.

\begin{proposition} Let $\Sigma$ be a bounded-posetal PR-structure. If $\rho^{\Sigma}(a,\top)=\{\mathbf{t}\}$ for every $a\in \mathsf{P}_\Sigma$, then  $\Sigma$ is a P-structure. 
\end{proposition}
\begin{proof}
If $\rho^{\Sigma}(a,\top)=\{\mathbf{t}\}$ for every $a\in \mathsf{P}_\Sigma$, then in particular $\rho^{\Sigma}(\top,\top)=\{\mathbf{t}\}$. However, since $\Sigma$ is preorderal, $\mathbf{t}\in \rho^{\Sigma}(a,a)$ for every $a\in \mathsf{P}_\Sigma$. 
Suppose now that $\rho^{\Sigma}(a,b)\neq \emptyset$ and $s\in \rho^{\Sigma}(a,b)$. Since $\mathbf{t}\in \rho^{\Sigma}(b,b)$, then $s\Box\mathbf{t}\in \rho^{\Sigma}(a,b)$. However, since $\mathbf{t}\in \rho^{\Sigma}(b,\top)$, $s\Box \mathbf{t}\in \rho^{\Sigma}(a,\top)=\{\mathbf{t}\}$.
Thus $s\Box \mathbf{t}=\mathbf{t}\in \rho^{\Sigma}(a,b)$. Thus we have proven that $\mathbf{t}$ is in $\rho^{\Sigma}(a,b)$ whenever $a\vdash^\Sigma  b$. As a consequence, $\Sigma$ is a P-structure. 
\end{proof}

Simmetrically, one has also the following

\begin{proposition}  Let $\Sigma$ be a bounded-posetal PR-structure. If $\rho^{\Sigma}(\bot,a)=\{\mathbf{b}\}$ for every $a\in \mathsf{P}_\Sigma$, then  $\Sigma$ is a P-structure. 
\end{proposition}

Here we have a sufficient condition for concluding that a partitioned posetal PR-structure is a P-structure.

\begin{proposition} Every partitioned posetal PR-structure such that $(I,\vdash^{\Sigma}_I)$ has a minimum for every $I$ (or such that $(I,\vdash^{\Sigma}_I)$ has a maximum for every $I$) is a P-structure.
\end{proposition}
\begin{proof} Let $\Sigma$ be a partitioned posetal PR-structure such that $(I,\vdash^{\Sigma}_I)$ has a minimum for every $I$. Suppose $r$ is in $R$. Then $\rho^{\Sigma}(a,b)=\{r\}$ for some $a,b$. By corollary \ref{part}, $\mathbf{i}\Box r=r$. Since $\rho^{\Sigma}(\bot,\bot)=\{\mathbf{i}\}$, then $\rho^{\Sigma}(\bot,a)=\{\mathbf{i}\}$ for every $a\in \mathsf{P}_\Sigma$ as a consequence of the proof of theorem \ref{bound}. Thus $\{\mathbf{i}\Box r\}=\rho^{\Sigma}(\bot,b)=\{\mathbf{i}\}$. Thus $r=\mathbf{i}$.
\end{proof}

\subsection{Bounded lattical PR-structures}
The category $\mathsf{bLat}$ has as objects bounded lattices and as arrows bounded lattice morphisms. A PR-structure is \emph{bounded lattical} if $\mathbf{p}_\Sigma$ factors through the subcategory  $\mathbf{J}:\mathsf{bLat}\rightarrow \mathsf{Bin}$, that is, for every set $I$, $\mathbf{p}_{\Sigma}(I)$ is a bounded lattice and for every function $f$, $\mathbf{p}_{\Sigma}(f)$ preserves finite suprema and infima.
As a direct consequence of the notions of binary infimum and supremum, and of proposition \ref{pointwise}, we have the following
\begin{proposition}
If  $\Sigma$ is a bounded-lattical PR-structure, then for every $\varphi,\psi:I\rightarrow \mathsf{P}_\Sigma$ and for every $i\in I$
\begin{enumerate}
\item $(\varphi\wedge_I\psi)(i)\vdash^\Sigma  \varphi(i)\wedge \psi(i)$
\item $\varphi(i)\vee \psi(i)\vdash^\Sigma (\varphi\vee\psi)(i)$
\end{enumerate}\end{proposition}
Next we show that, for a bounded-lattical PR-structure, the requirement to be finite does not force the fact that it is a P-structure. 
\begin{proposition} There exist finite bounded-lattical PR-structures which are not P-structures.
\end{proposition}
\begin{proof}
Take $\mathsf{P}_\Sigma=\{\bot,\top\}$, $\mathsf{R}_\Sigma=\{\mathbf{t},\mathbf{i},\mathbf{b}\}$ with $\bot\neq \top$ and $\mathbf{t},\mathbf{i},\mathbf{b}$ distinct and 
$$\begin{cases}
\rho(\bot,\bot):=\{\mathbf{b},\mathbf{i}\}\\
\rho(\bot,\top):=\{\mathbf{b},\mathbf{t}\}\\
\rho(\top,\top):=\{\mathbf{i},\mathbf{t}\}\\
\rho(\top,\bot):=\emptyset\\
\end{cases}$$
which can be represented as follows 
$\xymatrix{
\bot\ar@(ul,ur)^{\{\mathbf{b},\mathbf{i}\}}\ar[r]^{\{\mathbf{b},\mathbf{t}\}}	&\top\ar@(ul,ur)^{\{\mathbf{i},\mathbf{t}\}}\\
}$.
For every $I$, we have a bounded lattice $(\{\top,\bot\}^I,\vdash^\Sigma _I)$ of the following form.\\

{\tiny
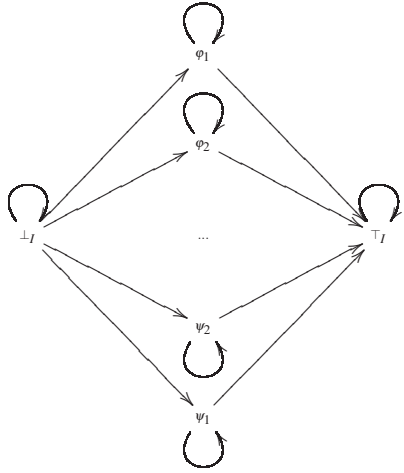
$$\xymatrix{
		&	&\varphi_1\ar@(ul,ur)	\ar[rrdd]		\\
		&	&\varphi_2\ar@(ul,ur)\ar[rrd]			\\
\bot_I\ar@(ul,ur)\ar[rruu]\ar[rru]\ar[rrd]\ar[rrdd]	& &...		&		&\top_I\ar@(ul,ur)\\
		&	&\psi_2\ar@(dl,dr)\ar[rru]\\
		&	&\psi_1\ar@(dl,dr)\ar[rruu]	& &	&\\
		&&&&\\ }$$
		}
		
Clearly binary infima and suprema are preserved by precomposition.

\end{proof}

\section{PR-structures coming from partial applicative structures}
We focus here on the case of PR-structures coming from partial applicative structures.
Let $\mathcal{R}=(|\mathcal{R}|,\cdot_\mathcal{R})$ be a partial applicative structure. We consider the PR-structure $\Sigma[\mathcal{R}]:=(\mathcal{P}(|\mathcal{R}|),|\mathcal{R}|,\Rightarrow^\mathcal{R})$ introduced above in Example 1.
Every partial applicative structure determines a function 
$$[\,]:|\mathcal{R}|\rightarrow \mathsf{Part}(|\mathcal{R}|,|\mathcal{R}|)$$
which sends each $r$ to the partial function $[r]$ of which the domain is the set $\mathsf{Dom}(r):=\{x\in |\mathcal{R}||\,r\cdot_\mathcal{R}x\downarrow\}$ and such that $[r](x)=r\cdot_{\mathcal{R}}x$ for every $x$ in the domain. We will denote with $\mathsf{Im}(r)$ the set $\{r\cdot x|\,x\in \mathsf{Dom}(r)\}$, that is the image of $[r]$. We also recall that a \emph{magma} is a partial applicative structure for which the binary partial function is total.

\noindent From now on, for sake of readability, we will omit subscripts and superscripts, and we will use $\mathcal{R}$ instead of $|\mathcal{R}|$.\\
\subsection{The preorderal and posetal cases}
First we prove that such PR-structures can never be non-trivial and partitioned.
\begin{proposition} Let $\mathcal{R}$ be a partial applicative structure. $\Sigma[\mathcal{R}]$ is partitioned if and only if $\mathcal{R}$ is a singleton.
\end{proposition}
\begin{proof}
This follows from the fact that, for any partial applicative structure $\mathcal{R}$, $(\emptyset\Rightarrow I)=
\mathcal{R} $ for every $I\subseteq 
\mathcal{R} $.
\end{proof}
Next we can use theorem \ref{preorderal} to characterize those $\Sigma[\mathcal{R}]$ which are preorderal.
\begin{proposition}
Suppose $\mathcal{R}$ is a partial applicative structure. The PR-structure $\Sigma[\mathcal{R}]$ is preorderal if and only if there exists $\mathbf{i}\in 
\mathcal{R} $ such that $\mathbf{i}\cdot  a=a$ for every $a\in 
\mathcal{R} $ and for every $r,s\in 
\mathcal{R} $, there exists $s\Box r\in 
\mathcal{R} $ such that for every $a\in 
\mathcal{R} $, if $ s\cdot  (r\cdot  a)\downarrow$, then $(s\Box r)\cdot  a= s\cdot  (r\cdot  a)$.
%
\end{proposition}
\begin{proof}
$(\Rightarrow)$ Suppose that the PR-structure associated to $\mathcal{R}$ is preorderal. Then there exists $\mathbf{i}\in 
\mathcal{R} $ such that $\mathbf{i}\in \rho(P,P)$ for every $P\subseteq 
\mathcal{R} $; in particular this holds for singletons, thus $\mathbf{i}\in \rho(\{a\},\{a\})$ for every $a\in 
\mathcal{R} $, that is $\mathbf{i}\cdot  a=a$. Moreover, for every $r,s\in 
\mathcal{R} $ there must be an element $s\Box r$ such that, for every $P,Q,R\subseteq 
\mathcal{R} $, if $r\in \rho(P,Q)$ and $s\in \rho(Q,R)$, then $s\Box r\in \rho(P,R)$; suppose that $a\in 
\mathcal{R} $ is such that $s\cdot  (r\cdot  a)\downarrow$. Then $s\Box r\in \rho(\{a\},\{s\cdot  (r\cdot  a)\})$, since $r\in \rho(\{a\},\{r\cdot  a\})$ and $s\in \rho(\{r\cdot  a\},\{s\cdot  (r\cdot  a)\})$. Thus $(s\Box r)\cdot  a=s\cdot  (r\cdot  a)$.

$(\Leftarrow)$ Suppose there exist such $\mathbf{i}$ and $s\Box r$ for every $r$ and $s$ and let us prove that they satisfy the requirements in the characterization of preorderal PR-structures. 
Let $P\subseteq 
\mathcal{R} $. For every $a\in P$ clearly $\mathbf{i}\cdot  a=a\in P$. Thus $\mathbf{i}\in \rho(P,P)$ for every $P\subseteq 
\mathcal{R} $. Suppose now that $P,Q,R$ are subsets of $
\mathcal{R} $. If $r\in \rho(P,Q)$ and $s\in \rho(Q,R)$ and $a \in P$, then $r\cdot  a\downarrow$ and $r\cdot  a\in Q$; thus $s\cdot (r\cdot  a)\downarrow$ and $s\cdot  (r\cdot  a)\in 
\mathcal{R} $. Since in this case $s\cdot  (r\cdot  a)=(s\Box r)\cdot  a$, we conclude that $s\Box r\in \rho(P,R)$.
%
%
%
%
\end{proof}
\begin{cor} If $\mathsf{R}_\Sigma $ is a partial applicative structure such that $\Sigma[\mathcal{R}]$ is preoderal, then $\mathcal{F} :=\{f:
\mathcal{R} \rightarrow 
\mathcal{R} |\,\textrm{there exists $r\in 
\mathcal{R} $ such that }f=[r]\}$
gives rise to a monoid together with composition of functions and the identity function. Moreover for every $r,s\in 
\mathcal{R} $, there exists $t\in 
\mathcal{R} $ such that $[t]\supseteq [s]\circ [r]$.
\end{cor}

\begin{theorem} If $\mathcal{R}$ is a partial applicative structure such that $\Sigma[\mathcal{R}]$ is posetal and $r\in 
\mathcal{R} $, then for every $s\in 
\mathcal{R} $, 
$r\cdot_\mathcal{R} s=s$ or there exists a natural number $n$ such that $[r]^n(s)\not\downarrow$ and $[r]^i(s)\neq [r]^j(s)$ for every $i,j<n$ such that $i\neq j$.
\end{theorem}

\begin{proof}
Let $r$ be an element of $
\mathcal{R} $ and consider the directed graph with loops determined by the function $[r]$, that is the one having as vertices the elements of $
\mathcal{R} $ and in which there is an edge from $x$ to $y$ if and only if $[r](x)=y$. Let us first take a look to the cycles in this graph. Let $x_1,...x_n$ be distinct vertices forming a cycle in the graph. Without loss of generality this means that $[r](x_i)=x_{i+1}$ for $i=1,...,n-1$ and $[r](x_n)=x_1$. In particular, this means that, since for every $i\in \mathbb{N}$ there exists $s\in 
\mathcal{R} $ with $[s]\supseteq [r]^i$, $\{x_i\}\vdash \{x_j\}$ and $\{x_j\}\vdash \{x_i\}$ for every $i,j=1,...,n$ from which it follows, by antisimmetry of $\vdash$, that $\{x_i\}=\{x_j\}$, that is $x_i=x_j$ for every $i,j=1,...,n$. Since we assumed that $x_1,...,x_n$ were distinct, then $n=1$, and the only possible cycle is a loop. 

Let us now consider the connected components of the graph obtained by not considering the directions of the edges. Let $X$ be a connected component. If $X$ contains a loop on $x$ there cannot be any other element $y\in X$; indeed if $[r](y)=x$, then $\{x,y\}\vdash \{x\}$ and $\{x\}\vdash \{x,y\}$ (since in the case of a preorderal $\Sigma[{\mathcal{R}}]$, the relation $\vdash$ is an extension of the inclusion relation $\subseteq$); by the antisymmetry of $\vdash$ we conclude that $\{x,y\}=\{x\}$, and hence that $x=y$. 
Thus the components containing a loop consist just of that loop.

Suppose that the component $X$ has no loops. We distinguish two cases:
\begin{enumerate}
\item $X$ contains at least a root $x$, that is a vertex such that $\delta_{out}(x)=0$, that is, in our case, $r\cdot  x\not\downarrow$. In this case $X$ contains exactly one root, since if there were two, $x$ and $x'$, then the path connecting them in the underlying non-directed graph would provide a vertex $y$ with $r\cdot_{\mathcal{R}} y$ having two distinct values, a contraddiction. In such a connected component every vertex $y$ is connected by a path of minimal lenght to $x$. This path is of the form $y$,$[r](y)$....,$[r]^{n-1}(y)=x$ for some $n\in \mathbb{N}$ and, by what we proved about loops and by minimality, $[r]^{i}(y)\neq [r]^j(y)$ for every $i\neq j$ with $i,j<n$; moreover $[r]^{n}(y)\not \downarrow$. 
\item $X$ contains no roots. Since $X$ is a tree, then $X$ is $2$-colourable. Consider the partition $(A,B)$ of the vertices of $X$ determined by a 2-coloration. Clearly $A,B\subseteq \mathsf{Dom}(r)$, $[r](A)\subseteq B$ and $[r](B)\subseteq A$, that is $A\vdash B$ and $B\vdash A$. Thus $A=B$. But $A\cap B=\emptyset$. Thus $X=\emptyset$, which is a contradiction. 
\end{enumerate}
We can hence conclude.\end{proof}
\begin{cor} If $\mathcal{R}$ is a partial applicative structure such that $\Sigma[\mathcal{R}]$ is posetal, $r\in 
\mathcal{R} $ and $[r]$ is total, then $[r]$ coincides with the identity function $\mathsf{id}_{
\mathcal{R} }$.
\end{cor}

\begin{cor} $\mathcal{R}$ is a magma such that $\Sigma[\mathcal{R}]$ is posetal if and only if $x\cdot_\mathcal{R}y=y$ for every $x,y\in 
\mathcal{R} $. In particular the PR-structure associated to $\mathcal{R}$ is a $P$-structure equivalent to that induced by the complete Boolean algebra $(\mathcal{P}(\mathcal{R}),\subseteq)$.
\end{cor}

\begin{cor} The unique partial combinatory algebra $\mathcal{R}$ such that $\Sigma[\mathcal{R}]$ is posetal is the trivial one.

\end{cor}
\begin{proof}
From $\mathbf{k}=\mathbf{i}$ it follows that $\mathbf{i}=(\mathbf{k}\cdot_\mathcal{R}\mathbf{i})\cdot_\mathcal{R}a=(\mathbf{i}\cdot_\mathcal{R}\mathbf{i})\cdot_\mathcal{R}a=a$ for every $a$.
\end{proof}

\begin{cor} If a partial applicative structure $\mathcal{R}$ admits a representation of pairs given by a pairing combinator $\mathbf{p}\in 
\mathcal{R} $ with projections $\mathbf{p}_0,\mathbf{p}_1\in 
\mathcal{R} $, that is 
\begin{enumerate}
\item $(\mathbf{p}\cdot_\mathcal{R}a_0)\cdot_\mathcal{R}a_1\downarrow$ for every $a_0,a_1\in 
\mathcal{R} $
\item $\mathbf{p}_i\cdot ((\mathbf{p}\cdot_\mathcal{R}a_0)\cdot_\mathcal{R}a_1)=a_i$ for every $a_0,a_1\in 
\mathcal{R} $, $i=0,1$,
\end{enumerate}
then $\mathcal{R}$ is trivial.
\end{cor}
\begin{proof}
For every $a_0,a_1\in 
\mathcal{R} $, $a_0=\mathbf{p}_0\cdot ((\mathbf{p}\cdot_\mathcal{R}a_0)\cdot_\mathcal{R}a_1)=a_0\cdot a_1=a_1$, since $[a_0]$ must be total. 
\end{proof}

Using theorem \ref{bound} we can prove the following:

\begin{proposition}
Let $\mathcal{R}=(
\mathcal{R} ,\cdot_\mathcal{R})$ be a partial applicative structure such that $\Sigma[\mathcal{R}]$ is posetal. Then $\Sigma[\mathcal{R}]$ is bounded posetal.
\end{proposition}
\begin{proof}
The minimum in $(\mathcal{P}(
\mathcal{R} ),\vdash)$ is $\emptyset$ and the maximum in $(\mathcal{P}(
\mathcal{R} ),\vdash)$ is $
\mathcal{R} $, since $\vdash$ extends $\subseteq$. The thesis follows by putting $\mathbf{b}=\mathbf{t}=\mathbf{i}$ in the statement of theorem \ref{bound} where $\mathbf{i}$ is such that $\mathbf{i}\cdot_{\mathcal{R}}x=x$ for every $x\in 
\mathcal{R} $.
\end{proof}

\subsection{On completeness of fibers of $\Sigma[\mathcal{R}]$}

We need first to give the following 

\begin{definition} Let $R$ be a binary relation on a set $A$ and let $(a_i)_{i\in I}$ be a set-indexed family of elements of $A$. An element $b\in A$ is 
\begin{enumerate}
\item a \emph{supremum} for $(a_i)_{i\in I}$ if 
\begin{enumerate}
\item $R(a_i,b)$ for every $i\in I$;
\item if $R(a_i,c)$ for every $i\in I$, then $R(b,c)$; 
\end{enumerate}
\item an \emph{adjoint-supremum} for $(a_i)_{i\in I}$ if for every $c\in A$
$$\left[R(a_i,c)\textnormal{ for every $i\in I$}\right]\textnormal{ if and only if }R(b,c).$$
\end{enumerate}
The binary relation $R$ is \emph{complete} (resp.\ \emph{adjoint-complete}) if every set-indexed family of elements of $A$ has a supremum (resp.\ adjoint supremum). 
\end{definition}
\begin{remark}
If $R$ is transitive and $b$ is a supremum for $(a_i)_{i\in I}$, then $b$ is also an adjoint-supremum. On the contrary, if $R$ is reflexive and $b$ is an adjoint-supremum for $(a_i)_{i\in I}$, then $b$ is also a supremum. In particular, if $R$ is a preorder on $A$, then the notions of supremum and adjoint-supremum coincide and, if they exist, they are unique up to isomorphism (that is, if $b$ and $b'$ are suprema of the same family, then $R(b,b')$ and $R(b',b)$).
\end{remark}

As we have already said, every element $r$ of a partial applicative structure naturally represents a function $[r]$ with domain $\mathsf{Dom}(r)$ and image $\mathsf{Im}(r)$ sending each $x$ to $r\cdot x$.
The following lemma shows that there is always a partial function of a certain kind which is not representable.
\begin{lemma}\label{lemmaext} Let $(\mathcal{R},\cdot)$ be a partial applicative structure and let $I$ be a set such that $|\mathcal{R}^I|>
\mathcal{R} $. Suppose $(X)_{i\in I}$ is a family of pairwise disjoint non-empty subsets of $\mathcal{R}$. Then, there exists a function $\varphi:\bigcup_{i\in I}X_i\rightarrow \mathcal{R}$ such that
\begin{enumerate}
\item for every $i \in I$ and every $x,y\in X_i$, $\varphi(x)=\varphi(y)$;
\item there is no $r\in \mathcal{R}$ such that for every $i\in I$ and $x\in X_i$, $r\cdot x\downarrow$ and $r\cdot x=\varphi(x)$.
\end{enumerate}
\end{lemma}
\begin{proof} The result follows immediately from 
$|\{\varphi:\bigcup_{i\in I}X_i\rightarrow \mathcal{R}|\,1.\textnormal{ holds}\}|=|\mathcal{R}^I|$.
\end{proof}

In $\mathsf{ZFC}$, every $I$ having cardinality greater than or equal to the cofinality $\mathsf{cf}(
\mathcal{R} )$ of the cardinality of $\mathcal{R}$ satisfies the hypothesis of the previous lemma. If in addition the generalized continuum hypothesis holds, then the two conditions are equivalent. 

\begin{definition} A partial applicative structure $\mathcal{R}$ is \emph{totally matching} if for every $x,y\in \mathcal{R}$ there exists $r$ such that $r\cdot x\downarrow $ and $r\cdot x=y$.
\end{definition}
\begin{remark} A totally matching non-trivial partial applicative structure can never give rise to a posetal PR-structure. Indeed, if $\mathcal{R}$ is \emph{totally matching}, then $\{x\}\vdash \{y\}$ for every pair of singletons.
\end{remark}
\begin{example} Every partial combinatory algebra $\mathcal{R}$ is totally matching: if $x, y \in \mathcal{R}$, then $(\mathbf{k}\cdot y)\cdot x=y$.
\end{example}
\begin{example} Every group $G$ is a totally matching partial applicative structure. Indeed, if $x,y\in G$, then $(yx^{-1})x=y$.
\end{example}

We are now ready to state the main result:
\begin{theorem} If $\mathcal{R}$ is a totally matching partial applicative structure such that there exists $r\in \mathcal{R}$ such that $|\mathcal{R}^{\mathsf{Im}(\{r\})}|>|\mathcal{R}|$, then $(\mathcal{P}(\mathcal{R})^{\mathcal{R}},\vdash_{\mathcal{R}})$ is not complete.
\end{theorem}
\begin{proof}
Let $r\in \mathcal{R}$ satisfy $|\mathcal{R}^{\mathsf{Im}(\{r\})}|>|\mathcal{R}|$. 
For every $a\in \mathsf{Dom}(r)$, consider the function $\varphi_a:\mathcal{R}\rightarrow \mathcal{P}(\mathcal{R})$ defined as follows:
$$
\varphi_a(x)=\begin{cases}
\{a\}\textnormal{ if }x=a\\
\emptyset\textnormal{ if }x\neq a\\
\end{cases}
$$
Suppose that $\psi$ is a supremum for the family $(\varphi_a)_{a\in \mathsf{Dom}(r)}$.

If we define the function $\mathsf{sgl}_r:\mathcal{R}\rightarrow \mathcal{P}(\mathcal{R})$ as follows
$$\mathsf{slg}_{r}(x)=\begin{cases}
\{r\cdot x\}\textnormal{ if }x\in \mathsf{Dom}(r)\\
\emptyset\textnormal{ if }x\notin \mathsf{Dom}(r)\\
\end{cases},$$
then $(\varphi_a\vdash_{\mathcal{R}} \mathsf{sgl}_r)$ for every $a\in \mathsf{Dom}(r)$ (just use $r$ itself as a realizer). Thus $\psi\vdash_{\mathcal{R}} \mathsf{slg}_r$. 
In particular, this implies that if $\psi(a)\cap \psi(b)\neq \emptyset$, then $r\cdot a=r\cdot b$, and that $\psi(a)=\emptyset$ whenever $a\notin\mathsf{Dom}(r)$. Moreover, since $\varphi_a\vdash_\mathcal{R}\psi$ for every $a\in \mathsf{Dom}(r)$, for those $a$ we have that $\psi(a)\neq \emptyset$.

For every $b\in \mathsf{Im}(r)$, we define $\psi'(b):=\bigcup_{\{a\in \mathcal{R}|\,r\cdot a\downarrow, r\cdot a=b\}}\psi(a)$ and we consider the family $(\psi'(b))_{b\in \mathsf{Im}(r)}$.
We are in the conditions for applying lemma \ref{lemmaext}. Thus there exists a function $\varphi:\bigcup_{b\in \mathsf{Im}(r)}\psi'(b)\rightarrow \mathcal{R}$ such that $\varphi(x)=\varphi(y)$ for every $b\in \mathsf{Im}(r)$ and for every $x,y\in \psi'(b)$, and for which there is no $s\in\mathcal{R}$ such that $s\cdot x\downarrow$ and $s\cdot x=\varphi(x)$ for every $b\in \mathsf{Im}(r)$ and $x\in \psi'(b)$. 

Let $\tilde{\varphi}:\mathcal{R}\rightarrow \mathcal{R}$ be the function defined by

$$\widetilde{\varphi}(a):=\begin{cases}\{\varphi(x)|\,x\in \psi(a) \}\textnormal{ if }a\in \mathsf{Dom}(r)\\ \emptyset \textnormal{ otherwise}\end{cases}$$

Since $\mathcal{R}$ is totally matching, then $\varphi_a\vdash_\mathcal{R} \widetilde{\varphi}$ for every $a\in \mathsf{Dom}(r)$.

From this it follows that $\psi\vdash_\mathcal{R}\widetilde{\varphi}$, that is there exists $s\in \mathcal{R}$ such that for every $a\in \mathsf{Dom}(r)$ and for every $x\in \psi(a)$, $s\cdot x=\varphi(x)$. This is a contraddiction.

\end{proof}

\begin{cor} If $\mathcal{R}$ is preorderal and total matching, then $(\mathcal{P}(\mathcal{R})^{\mathcal{R}},\vdash_{\mathcal{R}})$ is not complete. In particular this happens if $\mathcal{R}$ is a partial combinatory algebra: the triposes giving rise to realizability toposes do not factor through the inclusion of the category of complete pre-Heyting algebras in $\mathsf{Bin}$.
\end{cor}

\section{Conclusions} This is just the first step in a bottom-up investigation on PR-structures. Among all different directions of research connected with such a very general structure, there is at least one very interesting problem: as we have seen there are example of finite PR-structures giving rise to bounded lattical structures. Is there some minimal requirement expressed in terms of ``factorization through a category $\mathsf{C}$'' which guarantee that every finite $\mathsf{C}$-al PR-structure is a P-structure (in the posetal and in the non-posetal case)? For the posetal case, the category of distributive lattices seems to be a candidate, but this is just a conjecture.  

Another potentially interesting direction consists in the study of the relation between PR-structures and Miquel's implicative algebras (see \cite{Miquel2018ImplicativeAA}).

\bibliographystyle{plain}

\end{document}